\documentclass[12pt,reqno]{amsart}
\usepackage{amssymb,amsmath,amsthm,newlfont,enumerate,color}

\renewcommand{\phi}{\varphi}

\numberwithin{equation}{section}

\newtheorem{Proposition}[equation]{Proposition}

\newtheorem{definition}[equation]{Definition}
\newtheorem{theorem}[equation]{Theorem}
\newtheorem{Lemma}[equation]{Lemma}

\begin{document}
\bibliographystyle{amsplain}

    \title[Properties of unbounded  TTO.]{Some properties of unbounded  truncated Toeplitz operators.}

    \author[A. Chettih, A. Yagoub, Z. Bendaoud]{ Ali Chettih$^1$, Ameur Yagoub$^2$, Zohra Bendaoud$^3$ }
   
     \address{Ali Chettih. Universit\'e de M. Khider\\ Biskra\\ Alg\'erie 07000.}
    \email{a.chettih@ens-lagh.dz}
\address{Ameur Yagoub. Laboratoire de math\'ematiques pures et appliqu\'es\\ Universit\'e de Amar telidji Laghouat\\ Alg\'erie 03000.}
  \email{a.yagoub@lagh-univ.dz}
   \address{Zohra Bendaoud. Department of mathematics, Teachers' Higher college of Laghouat, Algeria, 03000.}
    \email{zbendaouad@gmail.com}
 \begin{abstract}
 In this paper, we study closed densely defined unbounded truncated Toeplitz operators on $K_u=(uH^2)^\perp$, where $u$ is an inner function, that commute with modified compressed shifts. The work also establishes properties related to their invertibility and self-adjointness.
 \end{abstract}

\maketitle

\section{Introduction}Truncated Toeplitz operators are the compressions of the usual Toeplitz operators on the Hardy space $H^2$ to its coinvariant subspace, the model space $K_u=(uH^2)^\perp$, where $u$ is an inner function. They were formally introduced by Sarason in \cite{sar}. The operators $S_u,S_u^*$, and $S_u^\alpha$ (with $\alpha\in \mathbb{C}\cup \lbrace\infty\rbrace)$ denote the compressed shift, backward compressed shift, and modified compressed shift operators on $K_u$, respectively. A bounded operator on $ K_u $ that commutes with $ S_u$ must be a truncated Toeplitz operator with a symbol in \( H^\infty \) (Sarason \cite{sar2}).    Sedlock \cite{sed2} established the analogous result for bounded operators on $K_u$ that commute with $S_u^\alpha$. For unbounded truncated Toeplitz operators, Su\'arez \cite{su} described the closed densely defined operators on $K_u$ that commute with $S_u^*$. Sarason \cite{sar3} later extended this work, demonstrating that $A=\varphi(S_u)$ for some $\varphi$ in a specific class of Nevanlinna functions.

In this article, we examine unbounded truncated Toeplitz operators that commute with $ S_u^\alpha $. We delineate the essential and sufficient conditions for the invertibility of these operators. We also explain when these operators are self-adjoint, including a description of their eigenvalues and eigenspaces.

In Section 2, we present notation and review certain findings regarding unbounded operators. We also present the multiplication operator and its properties. In Section 3, we present some results of papers of Sarason \cite{sar3,sar4}.  We study unbounded truncated Toeplitz operators commuting with  $S_u^\alpha$, in Section 4. Section 5 sets out the necessary and sufficient conditions for these operators to be invertible. Finally, we discuss unbounded self-adjoint truncated Toeplitz operators and their eigenvalues and eigenspaces in Section 6.
\section{Preliminaries}
Let $\mathbb{C}$ be the complex plane, $\mathbb{D}$ be the open unit disk, and $\mathbb{T}$ be the unit circle. Denote by $H(\mathbb{D})$ the space of functions holomorphic in $\mathbb{D}$. We consider $H^2$ to be the standard Hardy space on $\mathbb{D}$ and $H^\infty$ to be the space of analytic functions that are bounded on $\mathbb{D}$. The standard shift operator on $H^2$ is $Sf = zf$, and its adjoint is $S^*f = \frac{f-f(0)}{z}$. If $|u| = 1$ $a.e.$ on $\mathbb{T}$, then $u$ is an inner function. For each non-constant inner function $ u $, the model space is defined as $ K_u = H^2 \circleddash uH^2 $. It is important to observe that $ K_u^\infty = K_u \cap H^\infty $ constitutes a dense subspace of $ K_u $. The reproducing kernel in $K_u$ for a point $\lambda\in\mathbb{D}$ is the function 
$k_\lambda^u(z)=\frac{1-\overline{u(\lambda)}u(z)}{1-\overline{\lambda}z},  z\in\mathbb{T}.$
The space $K_u$ has the natural conjugation $C_uf=\widetilde{f}:=u\overline{zf}$.
We will call the compression of the shift operator $S$ to $K_u$ $S_u$, and the restriction of $S^*$ to $K_u$ $S^*_u$. If $\alpha$ is in $\mathbb{D}$, the modified compressed shifts are defined as follows: 
$$S_u^\alpha=S_u+\frac{\alpha}{1-\alpha\overline{u(0)}}k_0^u\otimes\widetilde{k_0^u}.$$
If $\alpha\in\mathbb{T}$, then the Clark operators on $K_u$ are $S_u^\alpha:= U_\alpha$. A groundbreaking result by Clark \cite{cl} shows that each $U_\alpha$ is a cyclic unitary operator and that every unitary, rank-one perturbation of $S_u$ is of the form $U_\alpha$.
The family of measures $ \lbrace\mu_\alpha, \alpha\in\mathbb{T}\rbrace $ are called the Clark measures for $ u $, the positive measures defined by the relation 
\begin{equation}\label{cla}
Re\frac{1+\overline{\alpha}u(z)}{1-\overline{\alpha}u(z)}= \int_{\mathbb{T}} \frac{\zeta+z}{\zeta-z}d\mu_\alpha(\zeta).
\end{equation}
We refer the reader to the surveys \cite{ps,s} for more details.
\begin{theorem}\label{cl}(Clark \cite{cl}). Let $\mu_\alpha$ be the only finite positive Borel measure on $\mathbb{T}$ that satisfies the form (\ref{cla}) for an inner function $u$.
\begin{equation}
    (V_\alpha f)(z)=(1-\overline{\alpha} u(z))\int \frac{f(\zeta)}{1-\overline{\zeta} z}d\mu_\alpha(\zeta),
\end{equation}
 is a single operator that goes from $L^2(\mu_\alpha)$ to $K_u$. Also, if
 $$ M_\zeta:L^2(\mu_\alpha)\rightarrow L^2(\mu_\alpha),\quad (M_\zeta f)(\zeta)= \zeta f(\zeta),$$
 then $$ V_\alpha M_\zeta V_\alpha^{-1} =U_{\beta_\alpha}, \quad \beta_\alpha=\frac{\alpha-u(0)}{1-\alpha\overline{u(0)}}.$$
\end{theorem}
The Cauchy integral of a function $ f $ in $ L^1 $ of $ \mathbb{T} $ is the function $ \mathbb{P} f $ defined in $ \mathbb{D} $ by
$$(\mathbb{P} f)(z) =\frac{1}{2\pi}\int_\mathbb{T}\frac{f(e^{i\theta})}{1 - e^{-i\theta}z}d\theta.$$
 If $f$ is in $L^2$, then $\mathbb{P} f$ equals $Pf$, the projection of $f$ onto $H^2$.  For $\varphi$ in $L^2$ ,
the Toeplitz operator on $H^2$ with symbol $\varphi$ , is defined by $$T_\varphi f= \mathbb{P}(\varphi f ), \textit{ with domain } \mathcal{D}(T_\varphi) = \lbrace f \in H^2, \mathbb{P}(\varphi f) \in H^2\rbrace. $$ 
It is a densely defined, closed operator, and bounded if and only if $\varphi$ is bounded. We define $\mathbb{P}_uf$ for $f$ in $L^1$ as the function in $H(\mathbb{D}) $ provided by
$$(\mathbb{P}_u f)(z) =\frac{1}{2\pi}\int_\mathbb{T}f(e^{i\theta})\frac{1-\overline{u(e^{i\theta})}u(z)}{1 - e^{-i\theta}z}d\theta.$$ If $f$ is in $L^2$, then $\mathbb{P}_u f$ equals $P_uf$, the projection of $f$ onto $K_u$.
The truncated Toeplitz operator $A^u_\varphi$ with symbol $\varphi$ is defined as $$A^u_\varphi f= \mathbb{P}_u(\varphi f )\textit{ with domain } \mathcal{D}(A^u_\varphi) = \lbrace f \in K_u, \mathbb{P}_u(\varphi f) \in K_u \rbrace.$$ The operator $A^u_\varphi$ is closed  and densely defined (since $K_u^\infty \subset \mathcal{D}(A^u_\varphi)$). If $\varphi\in L^2$ and $A^u_\varphi$ is bounded, then $(A^u_\varphi)^* = A^u_{\overline{\varphi}}$ and $A^u_\varphi$ is $C_u$-symmetric, meaning $C_uA^u_\varphi C_u = (A^u_\varphi)^*.$ If $\varphi\in H^\infty$, then $A^u_\varphi=0$ if and only if $\varphi\in  uH^\infty$. If $\varphi\in L^2$, then $A^u_\varphi$ is injective if and only if the inner factor of $\varphi$ and $u$ are relatively prime (we denote  by $ g.c.i.d.(u,\varphi ) = 1$).
The Sedlock classes $\mathcal{B}_u^\alpha,\alpha\in\widehat{\mathbb{C}}= \mathbb{C}\cup \infty$, are precisely the maximal subalgebras of truncated Toeplitz operators bounded, defined by  
$$\mathcal{B}_{u}^{\alpha} = \lbrace A^u_{\varphi+\alpha \overline{S_u \widetilde{\varphi}}+c},\varphi\in K_u,c\in\mathbb{C}\rbrace.$$
The algebras $\mathcal{B}_{u}^{\alpha}$ serve as the commutants of modified compressed shifts. The commutant of $S_{u}^{\alpha}$, denoted $\{S_{u}^{\alpha}\}'$, is defined as the collection of all bounded operators on $K_u$ that commute with $S_{u}^{\alpha}$. In the subsequent lemma, we encapsulate the principal characteristics of the classes $\mathcal{B}_{ {u}}^\alpha$ and their associations with $S_u^\alpha$.
\begin{Lemma}\cite{sed}\label{pro10} Let $u$ be a non-constant inner function. Then 
		\begin{enumerate}
\item $ A\in\mathcal{B}_{u}^{\alpha} $ if and only if $A^*\in\mathcal{B}_{u}^{1/\overline{\alpha}}\ $.
\item  If $|\alpha| < 1$, then $S_u^\alpha$ is a completely nonunitary contraction, and 
			$$\mathcal{B}_{u}^{\alpha} = \{S_{u}^{\alpha}\}'=\left\lbrace \Psi(S_u^\alpha)=A^u_{\frac{\Psi}{1-\alpha\overline{u}}},\Psi\in H^\infty\right\rbrace. $$ 
\item If $|\alpha| > 1$, then $$\mathcal{B}_{u}^{\alpha} = \{(S_u^{1/\overline{\alpha}})^*\}'=\left\lbrace \widehat{\Psi}(S_u^{1/\overline{\alpha}})^*= A^u_{\frac{\alpha\overline{\Psi}}{\alpha-u}},\Psi\in H^\infty\right\rbrace, $$
     where $\widehat{\Psi}(z)= \overline{\Psi(\overline{z})}$ in $\mathbb{T}$.
\item If $|\alpha| = 1$, then $S_{u}^{\alpha} $ is a unitary operator  and $$\mathcal{B}_{u}^{\alpha} = \{S_{u}^{\alpha}\}'=\left\lbrace \Phi(S_u^\alpha),\Phi\in L^\infty(\mu_\alpha)\right\rbrace,\quad \mu_\alpha \textit{Clark measure}. $$ 
Operators in $\mathcal{B}_{u}^{\alpha}$ are unitarily equivalent to multiplication operators $M_\Phi$ on $L^2(\mu_\alpha)$ induced by functions  $\Phi\in L^\infty(\mu_\alpha)$.
\item   Let $A$ and $B$  be bounded truncated Toeplitz operators. Then $AB$ is also a bounded truncated Toeplitz operator if and only if either $A$ or $B$ is a scalar multiple of the identity, or $A,B\in\mathcal{B}_{u}^{\alpha}$ for some $\alpha\in\widehat{\mathbb{C}}$. In the latter situation, $AB\in\mathcal{B}_{u}^{\alpha}$ as well.
\end{enumerate}
	\end{Lemma}
Let $b$ be a nonconstant function in the unit ball of $H^\infty$, but not an extreme point. The de Brange-Rovnyak space $\mathcal{H}(b)$ is the image of $H^2$ through the operator $(I-T_bT_{\overline{b}})^{1/2}$, which means $$\mathcal{H}(b)=(I-T_bT_{\overline{b}})^{1/2}(H^2).$$ It is generally established that if $\|b\|_{\infty} < 1$, then $\mathcal{H}(b)=H^2$, and if $b$ is an inner function, then $\mathcal{H}(b)=H^2\ominus bH^2$. By the general property of de Branges–Rovnyak spaces, we know that $\mathcal{H}(b)$ is $S^*$-invariant. This space is essential to our study. Indeed, every operator we consider in this article is defined in terms of it. The book \cite{sarb} has a lot of information about the spaces $\mathcal{H} (b)$.
The Nevanlinna class $\mathcal{N}$ comprises the holomorphic functions in $\mathbb{D}$ characterized by the relation $\varphi=\frac{\psi}{\chi}$,  $\psi, \chi \in H^\infty$ and $\chi\neq0$.
The Smirnov class $\mathcal{N}^+\subset \mathcal{N}$ comprises all functions $\varphi=\frac{\psi}{\chi}\in\mathcal{N}$ with $\chi$ outer. It is stated in \cite{sar4} that every nonzero function $\varphi$ in $\mathcal{N}^+$ admits a canonical representation, namely a unique decomposition $\varphi = \frac{b}{a}$, where $a,b\in H^\infty$, $a$  outer, $a(0) > 0$, and $|a|^2 + |b|^2 = 1$ $a. e$ on $\mathbb{T}$. In this context, $f \in \mathcal{H}(b)$ if and only if there exists a unique $g \in H^2$ such that $T_{\overline{b}} f = T_{\overline{a}} g$. As defined in \cite{sar3}, the local Smirnov class $\mathcal{N}^+_u$ by 
$$\mathcal{N}^+_u=\lbrace \varphi=\frac{\psi}{\chi}\in \mathcal{N}^+, g.c.i.d.(u,\chi ) = 1  \rbrace,$$
 equivalent by  canonical representation 
 $$\mathcal{N}^+_u=\lbrace \varphi=\frac{b}{va}\in \mathcal{N}^+,a, b \in H^\infty, v \textit{ inner },g.c.i.d.(v,b ) = 1, g.c.i.d.(v,u ) = 1  \rbrace.$$
 We will go over some concepts and basic ideas from the theory of unbounded operators on a Hilbert space. 
Think of an operator $ A $ that works on a complex Hilbert space $ \mathcal{H}$.  We will use the symbol $\mathcal{H}\times \mathcal{H}$ to mean $\mathcal{H}\oplus \mathcal{H}$ and think of it as the set of all $2\times 1$ column vectors whose components belong to  $\mathcal{H}$. We will summarize here some standard facts and we refer to standard references, e.g., \cite{B,E,k}, for proofs, further details and references
\begin{definition} Let $A$ and $B$ are  unbounded linear operator on $  \mathcal{H}$. 
\begin{enumerate}
\item the graph of $A$ is defined by
$$  \mathcal{G}( A)=\lbrace (f,Af) \in \mathcal{H}\oplus \mathcal{H},f\in \mathcal{D}(A)\rbrace.$$
\item If $\mathcal{D}(B)\subset \mathcal{D}(A)$ and $Af = Bf$ for every $f\in \mathcal{D}(B)$, then we say that $A$ is an extension of $B$. We say that $B \subset A$ in this case. Also, $B \subset A$ if and only if $\mathcal{G}(B) \subset \mathcal{G}(A)$.
\item An operator $A$ is said to be closed whenever its graph $\mathcal{G}( A)$ is a closed subset of  $\mathcal{H}\oplus \mathcal{H}$.
\item $A$ is closable if $\overline{\mathcal{G}( A)}$ is the graph of a closed operator. In this case, $\overline{A}$ denotes this operator and is called the closure of $A$.
\begin{eqnarray*}
 \mathcal{D}(\overline{A})&= & 	\lbrace x\in \mathcal{H}, \exists (x_n)_n\subset\mathcal{D}(A), x_n \rightarrow x, Ax_n \textit{ converge}\rbrace\\
\overline{A} x&= &\lim_{n} Ax_n, \forall x\in  \mathcal{D}(\overline{A}).
\end{eqnarray*}	 
\item If $A$ is a closed linear operator with domain $\mathcal{D}(A)$ that is dense in $\mathcal{H}$, then it commutes with a bounded operator $B$ on $\mathcal{D}(A)$ if $B(\mathcal{D}(A))$ is a subset of $\mathcal{D}(A)$ and $BAf = ABf$ for all $f$ in $\mathcal{D}(A)$.
\item An operator $A$ on $\mathcal{H}$ has a bounded inverse if  $\exists A^{-1}$ bounded operator,  such that $AA^{-1}=I$ on $\mathcal{H}$ and $A^{-1}A=I$ on $\mathcal{D}(A)$, with $A^{-1}(\mathcal{D}(A))\subset\mathcal{D}(A)$.
\item Let $A$ be an operator on $\mathcal{H}$ that is densely defined and has a domain $\mathcal{D}(A)$. The adjoint $A^*$ is defined as
\begin{eqnarray*}
  \mathcal{D}\left(A^{*}\right)&=&\left\{y \in \mathcal{H}: \exists u \in \mathcal{H} ,\langle A x, y\rangle=\langle x, u\rangle\right., \forall x \in \mathcal{D}(A)\} , \\
 \langle A x, y\rangle &=&\left\langle x, A^{*} y\right\rangle, \forall x \in \mathcal{D}(A) \text { and }\forall y \in \mathcal{D}\left(A^{*}\right).
\end{eqnarray*}
\item  An unbounded operator $A$ on $\mathcal{H}$ is called self-adjoint if $A^*=A$, which means that $ \mathcal{D}(A^*)=\mathcal{D}(A) $ and $A^*f=Af$ for every $f$ in $\mathcal{D}(A)$. 
\end{enumerate}
\end{definition} 
\begin{Lemma} \label{de}   Let $A$ and $B$ be densely defined unbounded linear operators on $  \mathcal{H}$. Then
\begin{enumerate}
\item  If $AB$ is densely defined then $B^*A^*\subset(AB)^*$, the two operators are equal if
$A$ is bounded on $\mathcal{H}$.
\item $\mathcal{G}( A^*)=\lbrace W\mathcal{G}( A)\rbrace^\perp$, where  $W:\mathcal{H} \oplus \mathcal{H} \rightarrow \mathcal{H} \oplus \mathcal{H}$ is an unitary operator, defined by $W(f \oplus g)=g \oplus-f$.
\end{enumerate}
\end{Lemma}
\begin{Lemma}\label{sim}
Let $\mathcal{H}_1$ and $\mathcal{H}_2$ be Hilbert spaces, and let $U$ be a unitary operator that goes from $\mathcal{H}_1$ to $\mathcal{H}_2$. If $A : \mathcal{D}(A) \rightarrow \mathcal{H}_1$ is an operator on $\mathcal{H}_1$, examine the operator on $\mathcal{H}_2$, $A_2 : \mathcal{D}(A_2) \rightarrow \mathcal{H}_2$ with $A_2 := UA_1U^{-1}$, $\mathcal{D}(A_2) := U\mathcal{D}(A_1)$, and  
$$\mathcal{G}( UA_1U^{-1})=\lbrace f\oplus g \in \mathcal{H}_2\oplus \mathcal{H}_2 :(U^{-1}f\oplus U^{-1}g)\in \mathcal{G}( A_1)\rbrace.$$
Then $A_2$ is  densely defined, closed , closable, or self-adjoint if and only if $A_1$ is alike.
\end{Lemma}
We are now going to talk about multiplication operators, which are very important to the progress of our outcomes.
\begin{definition}[Multiplication operator] Let $(\Omega, \mathcal{A}, \mu)$ be a measure space characterized by a $\sigma$-finite measure.  We define the essential range of a measurable function $\varphi:\Omega\rightarrow \mathbb{C}$ as 
$$ \varphi_{ess}(\Omega)=\lbrace \lambda \in \mathbb{C} :\forall \varepsilon > 0, \mu( \lbrace s\in \Omega:|\varphi(s)-\lambda|<\varepsilon\rbrace)> 0\rbrace,$$
and the associated multiplication operator $M_\varphi$ by
\begin{eqnarray*}
 \mathcal{D}(M_\varphi)&=&  \lbrace f \in L^2(\Omega,\mu) :\varphi f\in L^2(\Omega,\mu)\rbrace,\\
 M_\varphi f &=&\varphi f,\quad f\in \mathcal{D}(M_\varphi).
\end{eqnarray*}
\end{definition}
These operators satisfy the following basic properties
\begin{Lemma}\label{mult}\cite{B,E}. Let $\varphi$ a measurable function. Let $M_\varphi$ be the multiplication operator on $L^2(\Omega,\mu)$ with domain $\mathcal{D}(M_\varphi)$. Then 
\begin{enumerate}
\item  $(M_\varphi, \mathcal{D}(M_\varphi))$ is closed and densely defined.
\item  $M_\varphi$ is bounded (with $\mathcal{D}(M_\varphi) = L^2(\Omega,\mu))$ if and only if the
function $\varphi\in L^\infty(\Omega,\mu)$, and $\|M_\varphi\|=\|\varphi\|_\infty$. 
\item $M_{\varphi_1}=M_{\varphi_2}$, if and only if $\varphi_1=\varphi_2$ $\mu-$a.e.
\item $M_\varphi^*=M_{\overline{\varphi}}$ (in particular, $M_{\varphi}$ is self-adjoint $\Leftrightarrow\varphi$ is real-valued).
\item  The spectrum of $ M_\varphi $ is the essential range of $\varphi$.
\item $\lambda$ is an eigenvalue of $ M_\varphi\Leftrightarrow \mu( \lbrace \varphi^{-1}(\lambda)\rbrace)>0$. The analogous eigenspace is $L^2(\lbrace\varphi= \lambda\rbrace,\mu)$, which is the space of all square-integrable functions that are supported by the set $\lbrace\varphi= \lambda\rbrace$ and are defined $\mu$-almost everywhere.
\item The operator $ M_\varphi $ has a bounded inverse if and only if $0 \notin
\varphi_{ess}(\Omega)$. In that case, one has 
$M_{\varphi}^{-1}= M_r$ for $r : \Omega \to \mathbb{C}$ defined by $r(z) :=1/\varphi(z)$ if $\varphi(z) \neq0$,
$0$ if $\varphi(z) = 0$.
\end{enumerate}
\end{Lemma}

\section{Unbounded operators commuting with $S_u$ and $S_u^*$}
We begin this section with unbounded Toeplitz operators, which have been studied by Sarason \cite{sar3} and Su\'arez \cite{su}. Let $\varphi\in\mathcal{N}^+$, with $\varphi = \frac{b}{a}$.  The Toeplitz operator $T_\varphi$ with symbol $\varphi$ is defined as the operator that multiplies by $\varphi$ on the domain $\mathcal{D}(T_\varphi) = \{ f \in H^2 : \varphi f \in H^2 \} = aH^2$.
 The operator $T_\varphi$ is closed and densely defined, which means that its adjoint $T_\varphi^*$ is also closed and densely defined.
  The domain $ \mathcal{D}(T^*_\varphi)=\mathcal{H}(b)$, and the graph $  \mathcal{G}(T_\varphi^*)=\lbrace f\oplus g \in H^2\oplus H^2 :T_{\overline{b}}f=T_{\overline{a}}g\rbrace $. From the description of $ \mathcal{G}(T_\varphi^*)$, it follows that if $f$ is in $\mathcal{D}(T^*_\varphi)$, then $T^*_\varphi S^* f = S^*T^*_\varphi
f$. In what follows we denote $T_{\overline{\varphi}} =T_{\varphi}^* $. The reason for such a notation for $T_{\varphi}^* $ is explained in \cite{sar3}.

For truncated Toeplitz operators, the operator $T_{\overline{\varphi}}$ induces an operator on $K_u$, denoted by $A^u_{\overline{\varphi}}$, defined as $A^u_{\overline{\varphi}} = T_{\overline{\varphi}}/\mathcal{D}(T_{\overline{\varphi}}) \cap K_u$. For any $\varphi \in \mathcal{N}^+_u$, one can naturally define $\varphi(S_u) = ((va)(S_u))^{-1} b(S_u)$ as a closed operator on a dense domain in $K_u$.  There is a natural conjugation $C_u$ on the space $K_u$. We define $A^u_{\varphi}$ to be the transform of $A^u_{\overline{\varphi}}$ under the symmetry $C_u$. Thus, $\mathcal{D}(A^u_{\varphi})= C_u\mathcal{D}(A^u_{\overline{\varphi}})=\lbrace f  :C_uf\in \mathcal{H}(b)\cap K_u\rbrace$ and $A^u_\varphi C_uf=C_uA^u_{\overline{\varphi}}f$ for $f$ in $\mathcal{D}(A^u_{\overline{\varphi}})$.
In \cite{sar3}, Sarason extends the results of Su\'arez in \cite{su}, using a functional calculus approach
\begin{theorem} \cite{sar3}\label{sar}
$A$ is a closed operator that is densely defined in $K_u$, it commutes with $S_u$ ($S^*_u$) if and only if $A =\varphi(S_u)$. ( $A =\varphi^*(S^*_u)$, $\varphi^*(z)=\overline{\varphi(\overline{z})} $ ) where $\varphi\in \mathcal{N}^+_u$.
\end{theorem} 
 We list a few interesting facts below in Sarason papers \cite{sar4,sar3}, in this following lemma
 :
 \begin{Lemma}\label{sarason} Let $\varphi\in \mathcal{N}^+_u$, then
     \begin{enumerate}
\item  The operators $A^u_\varphi$ and $A^u_{\overline{\varphi}}$ are adjoints of each other, with 
$$  \mathcal{G}( A^u_\varphi )=\lbrace f\oplus g \in K_{u}\oplus K_{u} :A^u_{b}f=A^u_{va}g\rbrace, $$ and $$\mathcal{G}( A^u_{\overline{\varphi}} )=\lbrace f\oplus g \in K_{u}\oplus K_{u} :A^u_{\overline{b}}f=A^u_{\overline{va}}g\rbrace .$$
\item The operators $\widehat{A}^u_\varphi:\mathcal{D}(\widehat{A}^u_\varphi)=A^u_{va}K_u\to K_u$, and $\widehat{A}^u_{\overline{\varphi}}:\mathcal{D}(\widehat{A}^u_{\overline{\varphi}})=A^u_{\overline{va}}K_u\to K_u$ are closable, their closures are   $A^u_\varphi$ and $A^u_{\overline{\varphi}}$, respectively, and $(\widehat{A}^u_\varphi)^*=A^u_{\overline{\varphi}}$, $(\widehat{A}^u_{\overline{\varphi}})^*=A^u_\varphi$.
\item Let $\psi\in H^\infty$, then
 $A^u_{\overline{\psi}} A^u_{\overline{\varphi}}f =A^u_{\overline{\varphi}} A^u_{\overline{\psi}}f=  A^u_{\overline{\varphi}\overline{\psi}}f$ for $f$ in $\mathcal{D}(A^u_{\overline{\varphi}})$.
\item If $\varphi$ has the  representation $\varphi =\frac{\psi}{va}$, where $\psi$ is in $H^\infty$, $A^u_{\overline{\varphi}}  A^u_{\overline{va}} h=A^u_{\overline{\psi}} h$ for $h$ in $K_u$. Moreover $A^u_{\overline{\varphi}}$ is the closure of its restriction to $A^u_{\overline{va}}K_u $.
\item
 $ A^u_{\overline{\varphi}}f =A^u_{1/\overline{v}} A^u_{\overline{b}/\overline{a}}f$ for $f$ in $\mathcal{D}(A^u_{\overline{\varphi}})$.
 \item Let $\varphi_1$ and $\varphi_2$ be two nonzero functions in $\mathcal{N}^+_u$ . Then $A^u_{\overline{\varphi_1}} = A^u_{\overline{\varphi_2}}$ if and only if $u$ divides $\varphi_1-\varphi_2$.
\end{enumerate}
 \end{Lemma}

\section{Unbounded operators commuting with  $S_u^\alpha$}
We organize this section by classifying densely defined closed truncated Toeplitz operators that commute with $S_u^\alpha$ into four cases, depending on $\alpha$.

Let $\alpha=0$ or $\alpha=\infty$.  We have $S_u^0=S_u$, and  $S_u^\infty=S^*_u$. These two cases are studied in Sarason's articles \cite{sar4,sar3}. 

Let $\alpha\in\mathbb{D}$. Recall that $u_\alpha=\frac{u-\alpha}{1-\overline{\alpha}u}$, for $\alpha\in\mathbb{D}$ and $J_\alpha=M_{(1-|\alpha|^2)^{-1/2}(1-\overline{\alpha}u)}$ is an unitary map from $K_{u_\alpha}$ onto $K_{u}$. Note that $J^{-1}_\alpha=M_{(1-|\alpha|^2)^{1/2}(1-\overline{\alpha}u)^{-1}}$. It follows that $J^{-1}_\alpha S_u^\alpha J_\alpha=S_{u_\alpha} $ and $C_uJ_\alpha=J_\alpha C_{u_\alpha}$. By Proposition \ref{sim}, $A$ is commutes with $S_{u}^\alpha$, closed and densely defined operator in $K_{u}$, if and only if $B$ is commutes with $S_{u_\alpha}$, closed and densely defined operator in $K_{u_\alpha}$. Moreover $$A=J_\alpha BJ^{-1}_\alpha \text{ with domain}  \quad \mathcal{D}(A) = \lbrace f \in K_u : J^{-1}_\alpha f \in \mathcal{D}(B)\rbrace .$$
If $ \varphi $ is in $H^\infty$ we have $A^u_{\frac{\varphi}{1-\alpha\overline{u}}}=J_\alpha A^{u_\alpha}_\varphi J^{-1}_\alpha$  and $A^u_{\frac{\overline{\varphi}}{1-\overline{\alpha}u}} =J_\alpha A^{u_\alpha}_{\overline{\varphi}} J^{-1}_\alpha$ (see \cite{sed,sed2}). As we mentioned earlier, here is no obvious way to define the  $T_{\overline{\varphi}}$   for $\varphi\in\mathcal{N}^+$. Sarason justified and formalized the definition of $T_{\overline{\varphi}}$ in \cite{sar3}. Because his definitions are inherently tied to the operator's domain of $T^*_\varphi$, we proceed to give the following definitions:
$$A^u_{\frac{\varphi}{1-\alpha\overline{u}}} =J_\alpha A^{u_\alpha}_\varphi J^{-1}_\alpha \textit{ and } A^u_{\frac{\overline{\varphi}}{1-\overline{\alpha}u}} = J_\alpha A^{u_\alpha}_{\overline{\varphi}} J^{-1}_\alpha, \textit{ For } \varphi\in\mathcal{N}^+_{u_\alpha}.$$
\begin{Proposition}\label{pro2}Let  $\alpha\in\mathbb{D}$ and $\varphi\in\mathcal{N}^+_{u_\alpha}$.  Then 
\begin{enumerate}
\item the operator $A^u_{\frac{\varphi}{1-\alpha\overline{u}}}$ with domain $\mathcal{D}(A^u_{\frac{\varphi}{1-\alpha\overline{u}}}) = \lbrace f \in K_u : J^{-1}_\alpha f \in \mathcal{D}( A^{u_\alpha}_\varphi)\rbrace $, is closed and densely defined, commutes with $S_u^\alpha$, and with the graph  $$  \mathcal{G}(A^u_{\frac{\varphi}{1-\alpha\overline{u}}})=\lbrace f\oplus g \in K_u\oplus K_u :A^u_{\frac{b}{1-\alpha\overline{u}}}f=A^u_{\frac{va}{1-\alpha\overline{u}}}g\rbrace.$$
\item  The graph of $A^u_{\frac{\overline{\varphi}}{1-\overline{\alpha}u}} $ is    $$  \mathcal{G}(A^u_{\frac{\overline{\varphi}}{1-\overline{\alpha}u}})=\lbrace f\oplus g \in K_u\oplus K_u :A^u_{\frac{\overline{b}}{1-\overline{\alpha}u}}f =A^u_{\frac{\overline{va}}{1-\overline{\alpha}u}}g\rbrace. $$ 
\item  The operators $A^u_{\frac{\varphi}{1-\alpha\overline{u}}}$ and $A^u_{\frac{\overline{\varphi}}{1-\overline{\alpha}u}}$ are adjoints of each other.
\end{enumerate}
\end{Proposition}
\begin{proof}Let $\varphi\in\mathcal{N}^+_{u_\alpha}$. (1) By Theorem \ref{sar}, $A^{u_\alpha}_\varphi$ is commutes with $S_{u_\alpha}$, closed and densely defined operator in $K_{u_\alpha}$, on $\mathcal{D}( A^{u_\alpha}_\varphi)$.  Moreover, by part (1) of Lemma \ref{sarason},   $$  \mathcal{G}( A^{u_\alpha}_\varphi )=\lbrace f\oplus g \in K_{u_\alpha}\oplus K_{u_\alpha} :A^{u_\alpha}_{b}f=A^{u_\alpha}_{va}g\rbrace. $$
Since $ J_\alpha $ is a unitary map from $K_{u_\alpha}$ onto $K_u$, the operators $A^u_{\frac{\varphi}{1-\alpha\overline{u}}}$ and $A^{u_\alpha}_\varphi$  are unitarily equivalent. For the graph of $A^u_{\frac{\varphi}{1-\alpha\overline{u}}}$, we have
\begin{eqnarray*}
  \mathcal{G}(A^u_{\frac{\varphi}{1-\alpha\overline{u}}})&=&\lbrace f\oplus g \in K_u\oplus K_u : J^{-1}_\alpha  f\oplus J^{-1}_\alpha g\in \mathcal{G}( A^{u_\alpha}_\varphi )\rbrace \\
  &=&\lbrace f\oplus g \in K_u\oplus K_u : A^{u_\alpha}_b J^{-1}_\alpha f= A^{u_\alpha}_{va} J^{-1}_\alpha g\rbrace \\
 &=& \lbrace f\oplus g \in K_u\oplus K_u :A^u_{\frac{b}{1-\alpha\overline{u}}}f=A^u_{\frac{va}{1-\alpha\overline{u}}}g\rbrace.
\end{eqnarray*}
We now prove that, $A^u_{\frac{\varphi}{1-\alpha\overline{u}}}$ commutes with $S_u^\alpha$. For $f\in\mathcal{D}(A^u_{\frac{\varphi}{1-\alpha\overline{u}}})$ $$A^u_{\frac{\varphi}{1-\alpha\overline{u}}} S_u^\alpha f=J_\alpha A^{u_\alpha}_\varphi S_{u_\alpha}J^{-1}_\alpha f=J_\alpha S_{u_\alpha}A^{u_\alpha}_\varphi J^{-1}_\alpha f=S_u^\alpha A^u_{\frac{\varphi}{1-\alpha\overline{u}}}  f.$$
(2) By part (1) of  Lemma \ref{sarason}, we have $  \mathcal{G}( A^{u_\alpha}_{\overline{\varphi}} )=\lbrace f\oplus g \in K_{u_\alpha}\oplus K_{u_\alpha} :A^{u_\alpha}_{\overline{b}}f= A^{u_\alpha}_{\overline{va}}g\rbrace $, and the operators $A^u_{\frac{\overline{\varphi}}{1-\overline{\alpha}u}}$ and $A^{u_\alpha}_{\overline{\varphi}}$  are unitarily equivalent. We get that
\begin{eqnarray*}
  \mathcal{G}(A^u_{\frac{\overline{\varphi}}{1-\overline{\alpha}u}})&=&\lbrace f\oplus g \in K_u\oplus K_u : A^{u_\alpha}_{\overline{b}} J^{-1}_\alpha f= A^{u_\alpha}_{\overline{va}} J^{-1}_\alpha g\rbrace \\
 &=& \lbrace f\oplus g \in K_u\oplus K_u :A^u_{\frac{\overline{b}}{1-\overline{\alpha}u}}f= A^u_{\frac{\overline{va}}{1-\overline{\alpha}u}}g\rbrace.
\end{eqnarray*}
(3) Since $J_\alpha$ is unitary, the operators  $ A^{u_\alpha}_\varphi $ and $ A^{u_\alpha}_{\overline{\varphi}} $ are adjoints of one another. By part 2) of Lemma \ref{sarason},  it follows that $ J_\alpha A^{u_\alpha}_{\varphi} J^{-1}_\alpha $ and $ J_\alpha A^{u_\alpha}_{\overline{\varphi}} J^{-1}_\alpha $ are also adjoints. Consequently,  $A^u_{\frac{\varphi}{1-\alpha\overline{u}}}$ and $A^u_{\frac{\overline{\varphi}}{1-\overline{\alpha}u}}$ are adjoints of each other if and only if this holds. 
\end{proof}
\begin{Proposition}\label{pro}Let $\alpha\in\mathbb{D}$, $\psi\in H^\infty$ and $\varphi\in\mathcal{N}^+_{u_\alpha}$. Then
\begin{enumerate}
\item $A^u_{\frac{\overline{\psi}}{1-\overline{\alpha}u}}A^u_{\frac{\overline{\varphi}}{1-\overline{\alpha}u}} f
  =A^u_{\frac{\overline{\varphi}}{1-\overline{\alpha}u}}A^u_{\frac{\overline{\psi}}{1-\overline{\alpha}u}}f= 
   J_\alpha A^{u_\alpha}_{\overline{\psi\varphi}} J^{-1}_\alpha f$, for $f$ in
    $\mathcal{D}(A^u_{\frac{\overline{\varphi}}{1-\overline{\alpha}u}})$.
    \item $A^u_{\frac{\psi}{1-\alpha\overline{u}}} A^u_{\frac{\varphi}{1-\alpha\overline{u}}} f
  =A^u_{\frac{\varphi}{1-\alpha\overline{u}}} A^u_{\frac{\psi}{1-\alpha\overline{u}}}f= 
   J_\alpha A^{u_\alpha}_{\psi\varphi} J^{-1}_\alpha f$, for $f$ in
    $\mathcal{D}(A^u_{\frac{\varphi}{1-\alpha\overline{u}}})$. 
\end{enumerate}

\end{Proposition}
\begin{proof}   Since the two equalities are related via the $C_u$-transform, it is enough to establish the first one. Let $\psi\in H^\infty$ and $\varphi\in\mathcal{N}^+_{u_\alpha}$.
If $f \in \mathcal{D}(A^u_{\frac{\overline{\varphi}}{1-\overline{\alpha}u}})$ then $J_\alpha f\in \mathcal{D}( A^{u_\alpha}_\varphi)$, and by part 3) of Lemma \ref{sarason}, we have
\begin{eqnarray*}
    A^u_{\frac{\overline{\psi}}{1-\overline{\alpha}u}}A^u_{\frac{\overline{\varphi}}{1-\overline{\alpha}u}} f&=&J_{\alpha}A^{u_\alpha}_{\overline{\psi}} A^{u_\alpha}_{\overline{\varphi}} J_{\alpha}^{-1} f=J_{\alpha}A^{u_\alpha}_{\overline{\varphi}}A^{u_\alpha}_{\overline{\psi}} J_{\alpha}^{-1} f\\
    &=&A^u_{\frac{\overline{\varphi}}{1-\overline{\alpha}u}}A^u_{\frac{\overline{\psi}}{1-\overline{\alpha}u}}f=J_{\alpha}A^{u_\alpha}_{\overline{\psi\varphi}}  J_{\alpha}^{-1} f.
\end{eqnarray*}
\end{proof}
\begin{Proposition}Let $\alpha\in\mathbb{D}$. Let $\varphi_1$ and
 $\varphi_2$ be two nonzero functions in $\mathcal{N}^+_{u_\alpha}$ . Then
 \begin{enumerate}
\item  $A^u_{\frac{\overline{\varphi_1}}{1-\overline{\alpha}u}} =A^u_{\frac{\overline{\varphi_2}}{1-\overline{\alpha}u}}$ if and only if $u_\alpha$ divides $\varphi_1-\varphi_2$.
\item  $A^u_{\frac{\varphi_1}{1-\alpha\overline{u}}} = A^u_{\frac{\varphi_2}{1-\alpha\overline{u}}}$ if and only if $u_\alpha$ divides $\varphi_1-\varphi_2$.

\end{enumerate}

\end{Proposition}
\begin{proof}The two equalities are $C_u$-transforms of each other, therefore it is enough to show the first one.
Let $\varphi_1,\varphi_2\in\mathcal{N}^+_{u_\alpha}$ . Since $ J_\alpha $ is unitary, it follows that $A^u_{\frac{\overline{\varphi_1}}{1-\overline{\alpha}u}} =A^u_{\frac{\overline{\varphi_2}}{1-\overline{\alpha}u}}$ if and only if  $ A^{u_\alpha}_{\overline{\varphi_1}}=A^{u_\alpha}_{\overline{\varphi_2}} $, but by part 6) of Lemma \ref{sarason} the latter is true if and only if  $u_\alpha$ divides $\varphi_1-\varphi_2$.
\end{proof}
For $\varphi\in\mathcal{N}^+_{u_\alpha}$, let $\widehat{A}^u_{\frac{\varphi}{1-\alpha\overline{u}}}$ be the operator with domain $A^u_{\frac{va}{1-\alpha\overline{u}}}K_u$  given by $$\widehat{A}^u_{\frac{\varphi}{1-\alpha\overline{u}}}A^u_{\frac{va}{1-\alpha\overline{u}}}h=A^u_{\frac{b}{1-\alpha\overline{u}}}h, h\in K_u.$$
Denoted by $\widehat{A}^u_{\frac{\overline{\varphi}}{1-\overline{\alpha}u}}$, coincides
with its transform under the conjugation $C_u$,  $\widehat{A}^u_{\frac{\varphi}{1-\alpha\overline{u}}}: = C_u\widehat{A}^u_{\frac{\overline{\varphi}}{1-\overline{\alpha}u}}C_u$ with domain $\mathcal{D}(\widehat{A}^u_{\frac{\overline{\varphi}}{1-\overline{\alpha}u}}) =C_u\mathcal{D}(\widehat{A}^u_{\frac{\varphi}{1-\alpha\overline{u}}})$.
The proof of the following proposition follows the approach of Sarason \cite{sar3,sar4}, suitably adapted to our context.
\begin{Proposition}Let $\alpha\in\mathbb{D}$ and $\varphi\in\mathcal{N}^+_{u_\alpha}$, we have
\begin{enumerate}
\item $A^u_{\frac{va}{1-\alpha\overline{u}}}K_u$ is dense in $K_u $.
\item $\widehat{A}^u_{\frac{\varphi}{1-\alpha\overline{u}}} \subset(\widehat{A}^u_{\frac{\overline{\varphi}}{1-\overline{\alpha}u}})^{*}$ and $\widehat{A}^u_{\frac{\overline{\varphi}}{1-\overline{\alpha}u}} \subset(\widehat{A}^u_{\frac{\varphi}{1-\alpha\overline{u}}})^{*}$.
\item  $A^u_{\frac{\overline{\varphi}}{1-\overline{\alpha}u}} =(\widehat{A}^u_{\frac{\varphi}{1-\alpha\overline{u}}})^{*}$ and $A^u_{\frac{\varphi}{1-\alpha\overline{u}}} =(\widehat{A}^u_{\frac{\overline{\varphi}}{1-\overline{\alpha}u}})^{*}$.
\item The operators $A^u_{\frac{\varphi}{1-\alpha\overline{u}}} $ and $A^u_{\frac{\overline{\varphi}}{1-\overline{\alpha}u}}$ are  the
respective closures of $\widehat{A}^u_{\frac{\varphi}{1-\alpha\overline{u}}}$
and $\widehat{A}^u_{\frac{\overline{\varphi}}{1-\overline{\alpha}u}}$.
\end{enumerate}
\end{Proposition}
\begin{proof}1)  Because $g.c.i.d.(u_\alpha, va) = 1$ and   $J_\alpha$ is unitary, We have $u_\alpha H^2 \cap vaH^2 = \lbrace 0\rbrace$, therefore  $Ker(A^{u_\alpha}_{\overline{va}})=\lbrace 0\rbrace$. So, its adjoint $A^{u_\alpha}_{va}$ has a dense range, which means that $A^u_{\frac{va}{1-\alpha\overline{u}}}K_u$ is also dense in $K_u$.\\
2) Since the two inclusions are $C_u$-transforms of one another, proving the first suffices. This means that $\mathcal{G}(\widehat{A}^u_{\frac{\varphi}{1-\alpha\overline{u}}})$ and $W\mathcal{G}(\widehat{A}^u_{\frac{\overline{\varphi}}{1-\overline{\alpha}u}})$ are orthogonal to each other. For $h_1$ and $h_2$ in $K_{u}$, the vector $F_{1}=A^u_{\frac{va}{1-\alpha\overline{u}}} h_{1} \oplus A^u_{\frac{b}{1-\alpha\overline{u}}} h_{1}$  in $\mathcal{G}(\widehat{A}^u_{\frac{\varphi}{1-\alpha\overline{u}}})$, and the vector $F_{2}=A^u_{\frac{\overline{va}}{1-\overline{\alpha}u}} h_2\oplus A^u_{\frac{\overline{b}}{1-\overline{\alpha}u}}h_2$  in $\widehat{A}^u_{\frac{\overline{\varphi}}{1-\overline{\alpha}u}}$, from which it follows that
$$
\begin{aligned}
\left\langle F_{1}, W F_{2}\right\rangle & =\left\langle A^u_{\frac{va}{1-\alpha\overline{u}}} h_{1}, A^u_{\frac{\overline{b}}{1-\overline{\alpha}u}}h_2\right\rangle-\left\langle A^u_{\frac{b}{1-\alpha\overline{u}}} h_{1}, A^u_{\frac{\overline{va}}{1-\overline{\alpha}u}} h_2\right\rangle \\
& =\left\langle  A^u_{\frac{b}{1-\alpha\overline{u}}}A^u_{\frac{va}{1-\alpha\overline{u}}} h_{1}, h_2\right\rangle-\left\langle A^u_{\frac{va}{1-\alpha\overline{u}}}A^u_{\frac{b}{1-\alpha\overline{u}}} h_{1}, h_2\right\rangle \\
& =\left\langle A^u_{\frac{bva}{1-\alpha\overline{u}}} h_{1}, h_2\right\rangle-\left\langle A^u_{\frac{vab}{1-\alpha\overline{u}}} h_{1}, h_2\right\rangle=0,
\end{aligned}
$$
from which the result follows. Moreover $\widehat{A}^u_{\frac{\varphi}{1-\alpha\overline{u}}}$ is closable.\\
3) Since $C_uJ_\alpha=J_\alpha C_{u_\alpha}$, The two claims are $C_u$-transforms of each other, thus it will be enough to show that the first one is true. The typical vector in $\mathcal{G}(\widehat{A}^u_{\frac{\varphi}{1-\alpha\overline{u}}})$ equals $A^u_{\frac{v a}{1-\alpha\overline{u}}}h \oplus A^u_{\frac{b}{1-\alpha\overline{u}}}  h$ with $h$ in $K_{u}$. Hence, the vector $f \oplus g$ in $K_{u} \oplus K_{u}$ is in $\mathcal{G}((\widehat{A}^u_{\frac{\varphi}{1-\alpha\overline{u}}})^{*})$ if and only if, for all $h$ in $K_{u}$,
$$
\begin{aligned}
0 & =\left\langle f \oplus g, W\left(A^u_{\frac{v a}{1-\alpha\overline{u}}}h \oplus A^u_{\frac{b}{1-\alpha\overline{u}}}  h\right)\right\rangle \\
& =\left\langle f, A^u_{\frac{b}{1-\alpha\overline{u}}}  h\right\rangle-\left\langle g, A^u_{\frac{v a}{1-\alpha\overline{u}}}h\right\rangle \\
& =\left\langle A^u_{\frac{\overline{b}}{1-\overline{\alpha}u}} f-A^u_{\frac{\overline{va}}{1-\overline{\alpha}u}} g, h\right\rangle,
\end{aligned}
$$
if and only if $A^u_{\frac{\overline{b}}{1-\overline{\alpha}u}} f=A^u_{\frac{\overline{va}}{1-\overline{\alpha}u}} g$, and by proposition \ref{pro2},  $f \oplus g$ in $\mathcal{G}(A^u_{\frac{\overline{\varphi}}{1-\overline{\alpha}u}})$.\\
4) The two statements are proved similarly; thus, establishing the first. Since $\widehat{A}^u_{\frac{\varphi}{1-\alpha\overline{u}}}$ is closable, we have 
$$\mathcal{G}(A^u_{\frac{\varphi}{1-\alpha\overline{u}}}) =\mathcal{G}((A^u_{\frac{\overline{\phi}}{1-\overline{\alpha}u}})^*) =\mathcal{G}((\widehat{A}^u_{\frac{\varphi}{1-\alpha\overline{u}}})^{**})=\overline{\mathcal{G}(\widehat{A}^u_{\frac{\varphi}{1-\alpha\overline{u}}})}.$$
\end{proof}
Let $\alpha\in\mathbb{C}/\overline{\mathbb{D}}$.  If $ A$ is densely defined then, by (1) of proposition \ref{de}, we have $A$ commutes with $S_u^\alpha$, then $A^*$ commutes with $S_u^{1/\overline{\alpha}}$, and so the above results can be applied to $A^*$ to get similar results for $A$. In which case we have
\begin{Proposition}Let $\alpha\in\mathbb{C}/\overline{\mathbb{D}}$. 
$A$ is a closed operator densely defined in $K_u$ commutes with $S^\alpha_u$ if
and only if $A =A^u_{\frac{\alpha\overline{\varphi}}{\alpha-u}}$ where $\varphi\in \mathcal{N}^+_{u_{(1/\overline{\alpha})}}$ , with domain $\mathcal{D}(A^u_{\frac{\alpha\overline{\varphi}}{\alpha-u}}) = \lbrace f \in K_u : \frac{\alpha f}{\alpha-u} \in \mathcal{D}( A^{{u_{(1/\overline{\alpha})}}}_{ \overline{\varphi}})\rbrace $, and $$  \mathcal{G}(A^u_{\frac{\alpha\overline{\varphi}}{\alpha-u}})=\lbrace f\oplus g \in K_u\oplus K_u :A^u_{\frac{\alpha\overline{b}}{\alpha-u}}f=A^u_{\frac{\alpha\overline{va}}{\alpha-u}}g\rbrace .$$
\end{Proposition}
Let  $|\alpha| = 1$. Then, $S^\alpha_u=U_\alpha$ is unitarily equivalent to the multiplication operator $M_z(f)=z f$ on $L^2(\mu_\alpha)$. In fact, theorem \ref{cl} says that $V_\alpha$ is a unitary operator from $L^2(\mu_\alpha)$ to $K_u$, and we obtain $S^\alpha_u=V_\alpha M_zV_\alpha^{-1}$. The multiplication operator $M_\Phi$ is defined for a measurable function $\Phi$ on $\mathbb{T}$ as 
$$M_\Phi f =\Phi f,\quad f\in \mathcal{D}(M_\Phi)= \lbrace f \in L^2(\mu_\alpha) :\Phi f\in L^2(\mu_\alpha)\rbrace.$$
By Lemma \ref{mult}, $M_\Phi$ is a closed densely defined operator on $L^2(\mu_\alpha)$. We set $\Phi(S^\alpha_u)=V_\alpha M_\Phi V_\alpha^{-1}$ with domain $\{f\in K_u,\  V_\alpha^{-1}f\in \mathcal{D}(M_\Phi)\}$. The following result is likely known, for example see  \cite[Proposition 6.3]{y}.
\begin{Proposition}\label{10} Let  $|\alpha| = 1$.
 $A$ is a closed operator that is densely defined in $K_u$, it commutes with $S^\alpha_u$ if and only if $A =\Phi(S^\alpha_u)$ for some   measurable function  $\Phi:\mathbb{T}\mapsto\mathbb{C}$. 
\end{Proposition}
We conclude this section with the following result, obtained via functional calculus.
\begin{theorem} \label{fc}
   $A$ is a closed operator that is densely defined in $K_u$, it commutes with $S^\alpha_u$ it is natural to interpret $A$ via the functional calculus of $S^\alpha_u$, by: 
\begin{enumerate}
\item if $|\alpha| < 1$, then $ A= A^u_{\frac{\Phi}{1-\alpha\overline{u}}}= \Phi(S_u^\alpha)= ((va)(S^\alpha_u))^{-1} b(S^\alpha_u)$, with $\Phi\in \mathcal{N}^+_{u_\alpha}$.
\item If $|\alpha| > 1$, then  $$A=A^u_{\frac{\alpha \overline{\Phi}}{\alpha-u}} =\Phi^*((S_u^{ 1/\overline{\alpha}})^*)= ((va)^*((S_u^{ 1/\overline{\alpha}})^*))^{-1} b^*((S_u^{ 1/\overline{\alpha}})^*),$$ with $\Phi\in\mathcal{N} ^+_{u_{1/\overline{\alpha}}}$, and $\Phi^*(z)=\overline{\Phi(\overline{z}})$. 
\item If $|\alpha|=1$, then $ A=\Phi(S_u^\alpha)$, for some  measurable function $\Phi:\mathbb{T}\mapsto\mathbb{C}$.
\end{enumerate}
\end{theorem}
\begin{proof}$A$ is a closed operator that is densely defined in $K_u$, it commutes with $S^\alpha_u$ if
and only if one of the following cases is true\\
(1)if $|\alpha| < 1$, by Theorem \ref{sar}, a closed operator $A$ densely defined in $K_{u_\alpha}$ commutes with $S_{u_\alpha}$ if and only if  $A_\varphi^{u_\alpha}=\varphi(S_{u_\alpha}) = ((va)(S_{u_\alpha})^{-1} b(S_{u_\alpha})$ where $\varphi\in \mathcal{N}^+_{u_\alpha}$, since $J_\alpha$ is unitary, then 
\begin{eqnarray*}
    A^u_{\frac{\varphi}{1-\alpha\overline{u}}}&=&J_\alpha A^{u_\alpha}_\varphi J^{-1}_\alpha
 =J_\alpha\varphi(S_{u_\alpha})J^{-1}_\alpha\\
 &=&\varphi(J_\alpha S_{u_\alpha}J^{-1}_\alpha)
 =\varphi(S_u^\alpha)
  =((va)(S^\alpha_u))^{-1} b(S^\alpha_u).
\end{eqnarray*}
(2) If $|\alpha| > 1$, then apply to $A^*$ to obtain similar results for $A$.\\
(3) If $|\alpha|=1$, the proof follows immediately from Proposition \ref{10}. 
\end{proof}
\section{Inverse of unbounded  truncated Toeplitz operators}
Before presenting our main results on unbounded operators that admit a bounded inverse, we first review the corresponding results for bounded operators, as studied by Sedlock \cite{sed}.
\begin{theorem} \cite{sed} Let $A$ be an invertible truncated Toeplitz operators. Then $A^{-1}$ is a truncated Toeplitz operators if and only if $A\in\mathcal{B}_{u}^{\alpha}$. As a result, $A^{-1}\in\mathcal{B}_{u}^{\alpha}$.
    
\end{theorem}
We now state the following result concerning unbounded truncated Toeplitz operators.
\begin{theorem}Let $A$ is closed operator  densely defined in $K_u$ commutes with $S^\alpha_u$.    The operator $A$ has a bounded inverse if and only if one of the following conditions holds: 
\begin{enumerate}
\item  If $|\alpha| < 1$, then $ A= A^u_{\frac{\varphi}{1-\alpha\overline{u}}}$,  $\varphi\in \mathcal{N}^+_{u_\alpha}$,  $A^{- 1 }=A^u_{\frac{va\psi}{1-\alpha\overline{u}}}$,   $\psi\in H^\infty$, and  $u_\alpha$ divides $\psi b-1.$
\item If $|\alpha| > 1$, then $A=A^u_{\frac{\alpha \overline{\phi}}{\alpha-u}}$,  $\varphi\in \mathcal{N} ^+_{u_{1/\overline{\alpha}}}$,  $A^{- 1 }=A^u_{\frac{\alpha \overline{va\psi}}{\alpha-u}}$,  $\psi\in H^\infty$, and  $u_\alpha$ divides $\psi b-1.$
\item If $|\alpha|=1$, then $ A=\Phi(S_u^\alpha)$ for a measurable function $\Phi:\mathbb{T}\mapsto\mathbb{C}$,   $0 \notin
\Phi_{ess}(\mathbb{T})$, and 
$A^{- 1 }=\Psi(S_u^\alpha)$  for $\Psi : \mathbb{T} \to \mathbb{C}$ defined by $\Psi(z) :=1/\Phi(z)$ if $\Phi(z) \neq0$,
$0$ if $\Phi(z) = 0$. 
\end{enumerate}	

\end{theorem}
\begin{proof}
Let $f$ be a function in the domain $\mathcal{D}(A)$. Since $A$  commutes with $S^\alpha_u$, we have $S_u^\alpha(\mathcal{D}(A))\subseteq \mathcal{D}(A)$. Moreover, because $A^{-1}f \in \mathcal{D}(A)$, it follows that $S_u^\alpha A^{-1}f\in\mathcal{D}(A)$. Therefore,  $$S_u^\alpha f = S_u^\alpha A A^{- 1 } f=A S_u^\alpha A^{- 1 } f,$$ and thus $$A^{- 1 }S_u^\alpha f = A^{- 1 }A S_u^\alpha A^{- 1 } f=S_u^\alpha A^{- 1 } f.$$ so $$A^{- 1 }S_u^\alpha=S_u^\alpha A^{- 1 }\textit{ on } \mathcal{D}(A).$$ since $\mathcal{D}(A)$ is dense in $K_u$, $$A^{- 1 }S_u^\alpha=S_u^\alpha A^{- 1 } \textit{ on } K_u.$$ Then $A^{- 1 } $ is a bounded operator on $K_u$ which commutes with $S_u^\alpha$, hence $A^{- 1 } \in \mathcal{B}_u^\alpha$, and therefore by Theorem \ref{fc} and Lemma \ref{pro10}, there are three cases:\\
(1) If $|\alpha| < 1$, then $ A= A^u_{\frac{\varphi}{1-\alpha\overline{u}}}$, with $\varphi\in \mathcal{N}^+_{u_\alpha}$, and $A^{- 1 }=A^u_{\frac{\Psi}{1-\alpha\overline{u}}}$, with $\Psi\in H^\infty$. For $f$ in
    $\mathcal{D}(A^u_{\frac{\varphi}{1-\alpha\overline{u}}})$,  proposition \ref{pro} gives  $A^{- 1 }Af=A^u_{\frac{\Psi\varphi}{1-\alpha\overline{u}}}f=f$  if and only if $u_\alpha$ divides $\Psi\varphi-1$. Equivalently, if there exists $\psi\in H^\infty$ such that $\Psi=\psi va$, then  $u_\alpha$ divides $\psi va\varphi-1=\psi b-1.$ Let now, $f\in K_u$, we aim to show that $A^u_{\frac{va\psi}{1-\alpha\overline{u}}}f\in\mathcal{D}(A^u_{\frac{\varphi}{1-\alpha\overline{u}}})=\mathcal{D}((\widehat{A}^u_{\frac{\overline{\varphi}}{1-\overline{\alpha}u}})^{*})$. Suppose $g=A^u_{\frac{\overline{va}}{1-\overline{\alpha}u}}h\in\mathcal{D}(\widehat{A}^u_{\frac{\overline{\varphi}}{1-\overline{\alpha}u}})$, where $h$ in $K_{u}$, then 
$$
\begin{aligned}
\left\langle \widehat{A}^u_{\frac{\overline{\varphi}}{1-\overline{\alpha}u}}g,A^u_{\frac{va\psi}{1-\alpha\overline{u}}}f \right\rangle & =\left\langle A^u_{\frac{\overline{b}}{1-\overline{\alpha}u}}h,A^u_{\frac{va\psi}{1-\alpha\overline{u}}}f\right\rangle=\left\langle A^u_{\frac{\overline{va}}{1-\overline{\alpha}u}}h,A^u_{\frac{b\psi}{1-\alpha\overline{u}}}f\right\rangle\\
& =\left\langle g,A^u_{\frac{b\psi}{1-\alpha\overline{u}}}f\right\rangle,
\end{aligned}
$$
therefore $A^u_{\frac{va\psi}{1-\alpha\overline{u}}}f\in\mathcal{D}(A^u_{\frac{\varphi}{1-\alpha\overline{u}}}).$ Since $A^u_{\frac{\varphi}{1-\alpha\overline{u}}} $ closure of $\widehat{A}^u_{\frac{\varphi}{1-\alpha\overline{u}}}$, there exists $A^u_{\frac{va\psi}{1-\alpha\overline{u}}}f_n\in \mathcal{D}(\widehat{A}^u_{\frac{\varphi}{1-\alpha\overline{u}}}) $, converge to $A^u_{\frac{va\psi}{1-\alpha\overline{u}}}f$  such as $f_n\in  \mathcal{D}(\widehat{A}^u_{\frac{\varphi}{1-\alpha\overline{u}}}) $ and $\widehat{A}^u_{\frac{\varphi}{1-\alpha\overline{u}}}f_n=A^u_{\frac{b}{1-\alpha\overline{u}}}h_n$ where $f_n=A^u_{\frac{va}{1-\alpha\overline{u}}}h_n$ and $h_n\in K_u$, it follows that,   
\begin{eqnarray*}
  A^u_{\frac{\varphi}{1-\alpha\overline{u}}}A^u_{\frac{va\psi}{1-\alpha\overline{u}}}f &=&\lim_{n\to +\infty}\widehat{A}^u_{\frac{\varphi}{1-\alpha\overline{u}}}A^u_{\frac{va\psi}{1-\alpha\overline{u}}} f_n=\lim_{n\to +\infty}\widehat{A}^u_{\frac{\varphi}{1-\alpha\overline{u}}}A^u_{\frac{va\psi}{1-\alpha\overline{u}}} A^u_{\frac{va}{1-\alpha\overline{u}}}h_n\\&=&\lim_{n\to +\infty}\widehat{A}^u_{\frac{\varphi}{1-\alpha\overline{u}}}A^u_{\frac{va}{1-\alpha\overline{u}}} A^u_{\frac{\psi va}{1-\alpha\overline{u}}}h_n=\lim_{n\to +\infty}A^u_{\frac{b}{1-\alpha\overline{u}}}A^u_{\frac{\psi va}{1-\alpha\overline{u}}}h_n\\  
 &=&\lim_{n\to +\infty}A^u_{\frac{b\psi}{1-\alpha\overline{u}}} f_n=A^u_{\frac{b\psi}{1-\alpha\overline{u}}} f=A^u_{\frac{1}{1-\alpha\overline{u}}} f=f.
\end{eqnarray*}
(2)  For the case $|\alpha| > 1$, the corresponding results for $A$ follow by applying the preceding argument to its adjoint $A^*$.\\
(3) If $|\alpha|=1$, then $ A=\Phi(S_u^\alpha)$ for some  measurable function $\Phi:\mathbb{T}\mapsto\mathbb{C}$ has a bounded inverse if and only if $0 \notin
\Phi_{ess}(\mathbb{T})$. In that case, one has 
$A^{- 1 }=\Psi(S_u^\alpha)$  for $\Psi : \mathbb{T} \to \mathbb{C}$ defined by $\Psi(z) :=1/\Phi(z)$ if $\Phi(z) \neq0$,
$0$ if $\Phi(z) = 0$. 
\end{proof}
\section{Unbounded self-adjoint truncated Toeplitz operators}
If a truncated Toeplitz operator of type $ \alpha $ is self-adjoint, then necessarily $ |\alpha| = 1 $. In this case, the operator is of the form $\Phi(U_\alpha)$ for some $\Phi\in L^\infty(\mathbb{T},\mu_\alpha)$ that is real-valued $\mu_\alpha$-almost everywhere. For the unbounded case, we have the following result.
\begin{theorem}
$A$ is a closed operator that is densely defined in $K_u$, it commutes with $S^\alpha_u$ is self-adjoint if and only if $ A=\Phi(S_u^\alpha)$ for some $\alpha\in\mathbb{T}$ and a real-valued measurable function  $\Phi$ on $ \mathbb{T}$.
\end{theorem}
\begin{proof}
The proof follows from Proposition \ref{mult} and Theorem \ref{fc}.
\end{proof}

\subsection*{Eigenvalues and eigenspaces}
We assume that $\alpha$ is in $\mathbb{T}$ in this part. So, a closed operator $A$ that is densely defined in $K_u$, commutes with $S^\alpha_u$, and is self-adjoint can be expressed as  $ A=\Phi(S_u^\alpha)=V_\alpha M_\Phi V_\alpha^{-1}$ 
 for some real-valued measurable function  $\Phi$ on $ \mathbb{T}$. Since $V_\alpha$  is unitary, the eigenvalues of $ \Phi(S_u^\alpha) $ coincide with those of $M_\Phi$, and the corresponding eigenspaces are related via $V_\alpha$.
\begin{Proposition}
Let $A$ is closed, operator  densely defined operator in $K_u$ commutes with $S^\alpha_u$ and self-adjoint. Then 
\begin{enumerate}
\item The operator $ \Phi(S_u^\alpha) $ has a bounded inverse if and only if $0\notin \Phi_{ess}(\mathbb{T})$. In that case, one has $ \Phi^{-1}(S_u^\alpha)= \Psi(S_u^\alpha) $
for $\psi : \mathbb{T} \to \mathbb{R}$ defined by $\psi(s)=\frac{1}{\Phi(s)}$ if $ \Phi(s)\neq 0 $ and $0$ otherwise.
\item The eigenvalues of $ \Phi(S_u^\alpha) $ are the $\lambda\in\mathbb{R}$  such that $\mu_\alpha( \lbrace \Phi^{-1}(\lambda)\rbrace)>0.$
 The  eigenspace is $\lbrace V_\alpha f, f\in L^2(\mathbb{T},\mu_\alpha), f/_{\lbrace \Phi\neq \lambda \rbrace}=0\rbrace.$
\end{enumerate}

\end{Proposition}
\begin{proof}Let $A$  is  closed, densely defined operator in $K_u$ commutes with $S^\alpha_u$ and self-adjoint. \\
(1) The proof follows immediately from propositions \ref{10} and Lemma \ref{mult}. \\
(2)     By Lemma \ref{mult}, a real number $ \lambda $ is an eigenvalue of $ \Phi(S_u^\alpha) $ if and only if $ \lambda $ is an eigenvalue of $M_\Phi$ if and only if $ \mu_\alpha( \lbrace \Phi^{-1}(\lambda)\rbrace)>0$. The associated eigenspace is  $$f\in L^2(\lbrace\Phi= \lambda\rbrace,\mu_\alpha)=\lbrace f\in L^2(\mathbb{T},\mu_\alpha), f/_{\lbrace \Phi\neq \lambda \rbrace}=0\rbrace.$$ We have  $M_\Phi= V_\alpha^{-1}\Phi(S_u^\alpha)V_\alpha  $ , then  $f\in L^2(\lbrace\Phi= \lambda\rbrace,\mu_\alpha),$ is eigenspace of $M_\Phi$ if and only if $V_\alpha f$  is eigenspace of $ \Phi(S_u^\alpha) $.

\end{proof}

\section*{Acknowledgments} Author $^2$. I would like to dedicate this work to the memory of my supervisor,  Mohamed Zarrabi. It was as his student at IMB Bordeaux that he first introduced me to truncated Toeplitz operators—a moment that set the course of my entire mathematical career.\\
Authors $^{1,2,3}$. The General Direction of Scientific Research and Technological Development (DGRSDT) of Algeria is funding this research.


\begin{thebibliography}{9}
\thispagestyle{empty}
 \bibitem {B} D. Borthwick, Spectral theory: Basic concepts and applications, Graduate Texts in
Mathematics, vol. 284, Springer, Cham, 2020.
\bibitem {cl} D.N. Clark, One-dimensional perturbations of restricted shifts, J. Anal. Math. 25, 169–191, 1972.
\bibitem {E} K.J. Engel, and R. Nagel. One-parameter semigroups for linear evolution equations. Graduate
Texts in Mathematics, Vol 194. Springer-Verlag, New York, 2000.
\bibitem {ni2} N. Nikolski, Treatise on the shift operator. Spectral function theory,  Grundlehren derMathematischen Wissenschaften, vol. 273, Springer-Verlag, Berlin, 1986.
\bibitem {ps} A. Poltoratski, D. Sarason, Aleksandrov-Clark measures, recent advances in operator-related function theory, Contemp. Math., vol. 393, Amer. Math. Soc., Providence, RI, pp. 1–14, 2006. 
\bibitem {s} E. Saksman, An elementary introduction to Clark measures, Topics in complex
analysis and operator theory, Univ. M\'alaga, pp. 85–136, 2007.
\bibitem{sar}
D. Sarason, Algebraic properties of truncated Toeplitz operators. Oper. Matrices Vol.1 No.4, 491-526, 2007.
\bibitem{sar2}  D. Sarason, Generalized interpolation in $H^\infty$, Trans. Amer. Math. Soc 127, 179-203, 1967.
\bibitem{sarb}
 D. Sarason, Sub-Hardy Hilbert spaces in the unit disk, John Wiley, Sons Inc., New York, 1994.
 \bibitem{sar4} D. Sarason, Unbounded Toeplitz operators, Integral Equations Oper. Theory, 61, no. 2, 281-298. 2008. 
\bibitem{sar3} D. Sarason. Unbounded operators commuting with restricted backwards shifts. Oper. Matrices, Vol.2 No.4, 583-601, 2008.
\bibitem{sed2} N. A. Sedlock, Algebras of truncated Toeplitz operators , Oper. Matrices, Vol.5 No.2, 309-326, 2011.
\bibitem{sed} N. A. Sedlock, Properties of Truncated Toeplitz Operators, ProQuest LLC, Ann Arbor, MI, Ph.D. thesis, Washington University in St. Louis, 2010.
\bibitem {k} K. Schmudgen, Unbounded self-adjoint operators on Hilbert space. Graduate Texts in Math. 265, Springer, Cham, 2012.
\bibitem{su} D. Su\'arez. Closed commutants of the backwards shift operator. Pacific J. Math.,179, 371-396, 1997.
\bibitem{y} A. Yagoub, M. Zarrabi, Semigroups of truncated Toeplitz operators, Oper. Matrices, Vol.12 No.3, 603-618, 2018.
\end{thebibliography}
\end{document}